\documentclass[12pt]{article}

\usepackage{amsfonts}
\usepackage{amsmath}
\usepackage{amssymb}
\usepackage{amsthm}

\def\al{\alpha}

\def\om{\omega}
\def\Om{\Omega}
\def\hOmn{\hat{\Omega}_m}
\def\hOmnp{\hat{\Omega}_{m, \pi}}
\def\hPn{\hat{P}_m}
\def\hPnp{\hat{P}_{m,\pi}}

\def\Ph1{P^{(h_1)}}
\def\Ph2{P^{(h_2)}}

\def\b0{{\bf 0}}

\def\bnu{\bar\nu}

\def\1lbnu{{\bnu}_{\lambda_1}}
\def\2lbnu{{\bnu}_{\lambda_2}}

\newtheorem{thm}{Theorem}[section]
\newtheorem{prop}[thm]{Proposition}

\theoremstyle{plain}

\begin{document}

\title{A BK inequality for randomly drawn subsets of fixed size}

\author{J. van den Berg\footnote{CWI and VU University Amsterdam} 
and J. Jonasson\footnote{Chalmers University of Technology and Gothenburg University} \\
{\footnotesize email: J.van.den.Berg@cwi.nl, jonasson@chalmers.se}
}

\date{}
\maketitle

\begin{abstract}
The BK inequality (\cite{BK85}) says that,
for product measures on $\{0,1\}^n$, the probability
that two increasing events $A$ and $B$ `occur disjointly' is at most the product of the two individual
probabilities. The conjecture in \cite{BK85} that this holds for {\em all} events was proved in \cite{R00}. 

Several other problems in this area remained open.
For instance, although it is easy to see that non-product measures cannot satisfy the above
inequality for {\em all} events,  
there are several such measures which, intuitively, should satisfy the inequality for all 
{\em increasing} events. One of the most natural
candidates is
the measure assigning equal probabilities to all
configurations with exactly $k$ $1$'s (and probability $0$ to all other configurations).

The main contribution of this paper is a proof for these measures.
We also point out how our result extends to weighted versions of these measures, and to
products of such measures.

\end{abstract}

\begin{section}{Introduction and statement of results}

We start with an illustrative example, where two persons have to divide a random
collection of resources: 

\smallskip\noindent
{\bf Example.}
A large box contains thirty items, which we simply name by the numbers $1, \cdots, 30$.
Two persons, Alice and Bob, both have a list of those subsets of $\{1, \cdots, 30\}$ that he or she 
considers `useful'.
For instance, the items could be tools, and Alice may need to do a certain job which can
be performed with the combination of tools $\{5, 16,20 \}$, but can also be performed with the
combination $\{5, 18, 20\}$, the
combination $\{8, 25\}$ etcetera.
Similarly for Bob.

Now suppose that a fixed number, say, ten, of items is drawn randomly (uniformly) from the box.
These ten items are to be divided between Alice and Bob. For both persons to be satisfied, there must
be a pair of disjoint subsets of the ten items, such that one of these subsets is on Alice's
list, and the other is on Bob's list.
Our main result, Theorem \ref{bkrx}, says that the probability of this event is
at most the product of the probability that the above set of ten items
contains a set on Alice's list and the probability of the analogous event concerning Bob's list.

\begin{subsection}{Definitions, statement of results, and background} \label{sect-defs}

Before we state our main result,
we recall results from the literature which motivate our current work and are used in our proofs.
First some notation and definitions:
Throughout this paper, $\Om$ will denote the set $\{0,1\}^n$, and $P_p$ the product distribution on $\Om$ with parameter $p$.
We will often use the notation $[n]$ for $\{1, \cdots,n\}$.
For $\om \in \Om$ and $S \subset [n]$, we define $\om_S$ as the `tuple' $(\om_i, i \in S)$.
Further we use the notation $[\om]_S$ for the set of all elements of $\Om$ that `agree with $\om$ on $S$'. More formally,
$$[\om]_S := \{\al \in \Om \, : \, \al_S = \om_S\}.$$

Now, for $A, B \subset \Om$, $A \square B$ is defined as the event that $A$ and $B$ `occur disjointly' in the sense
that
there are disjoint subsets $K, L \subset [n]$ such that, informally speaking,
the $\om$ values on $K$ guarantee that
$A$ occurs, and the $\om$ values on $L$ guarantee that $B$ occurs. Formally, the definition is:

$$A \square B = \{\om \in \Om \, : \, \exists \text{ disjoint } K, L \subset [n] \mbox{ s.t. } 
[\om]_K \subset A \text{ and } [\om]_L \subset B\}.$$

For $\om$ and $\om' \in \Om$ we write $\om' \geq \om$ if $\om'_i \geq \om_i$ for all $i \in [n]$.
An event $A \subset \Om$ is said to be increasing if $\om' \in A$ whenever $\om \in A$ and $\om' \geq \om$. \\

Inequality \eqref{bkr-eq} below
was proved for increasing events in \cite{BK85}. That special
case has become a widely used tool in percolation theory and related topics (see e.g. \cite{G99} and \cite{G10}).
The paper \cite{BK85} also stated the conjecture that \eqref{bkr-eq} holds for all events.
Some other special cases were proved in \cite{BF87} and \cite{T94}.
There was not much hope for a proof of the general case until finally
this was obtained by the then unknown young mathematician D. Reimer, see \cite{R00}:

\begin{thm} \label{bkr}
For all $n$ and all $A, B \subset \{0,1\}^n$,
\begin{equation}\label{bkr-eq}
P_p(A \square B) \leq P_p(A) P_p(B).
\end{equation}
\end{thm}

\medskip
It is easy to see that
non-product measures on $\{0,1\}^n$, cannot satisfy (the analog of) \eqref{bkr-eq} for all events.
However, one may intuitively expect that many measures 
do satisfy the analog of \eqref{bkr-eq} for all {\em increasing} events. Such measures are said to have the BK property
(or, simply, to be BK measures). The most intuitively appealing
case where one may expect this property to hold, is the measure corresponding with randomly,
uniformly, drawing a subset of fixed size from the set $[n]$. (See
Section 3.1 of \cite{G94}, and the lines below
(4.18) in \cite{G10} where this has been conjectured).
It seems (oral communication) that several researchers have made efforts to prove this.

To be precise,
let $k \leq n$ and let $\Om_{k,n}$ be the set of all
$\om \in \Om$ with exactly $k$ $1's$.
Let $P_{k,n}$ (which we call the $k$-out-of-$n$ measure) be the distribution on $\Om$ that assigns
equal probability to all $\om \in \Om_{k,n}$ and probability $0$ to all other elements of $\Om$. 
Our main result, Theorem \ref{bkrx} below, is that such measures
indeed have the BK property. As far as we know, this is the first substantial example of a non-product
BK measure.

\begin{thm} \label{bkrx}
For all $n$, all $k \leq n$, and all increasing $A, B \subset \{0,1\}^n$,
\begin{equation}\label{bkrx-eq}
P_{k,n}(A \square B) \leq P_{k,n}(A) P_{k,n}(B).
\end{equation}
\end{thm}

\noindent
{\bf Remark:}
In Section \ref{sect-ext} we explain that this result and its proof extend to certain weighted versions of 
$P_{k,n}$ (also called conditional Poisson measures) and to products of such measures.

\smallskip
The rest of the paper is organized as follows:
In Section \ref{sect-reimer} we state Proposition \ref{reimer}, an intermediate result by Reimer which was of
crucial importance in his proof of Theorem \ref{bkr} (and which is also very interesting in itself).
In Section \ref{sect-proof-bkrx} we first state and prove Proposition \ref{reimerx}. This is an
analog of (and its proof uses) the above mentioned Proposition \ref{reimer}. Then we derive Theorem \ref{bkrx}
from Proposition \ref{reimerx}
in a way similar to that in which Reimer derived Theorem \ref{bkr} from Proposition \ref{reimer}.
We end the current section with some remarks which are of general interest but are not necessary for understanding the 
proof of Theorem  \ref{bkrx}.

\medskip\noindent
{\bf Remarks and discussion:} \\
(a) The example in the beginning of this section corresponds with the case 
$n = 30$ and $k = 10$ in Theorem \ref{bkrx}: Take 
$$A = \{ \om \in \Om \, : \,\, \mbox{supp}(\om) \mbox{ contains a set on Alice's list}\},$$  
where $\mbox{supp}(\om) = \{i \in [n] \, : \, \om_i =1\}$. Take $B$ similarly, but now with Bob's list. \\
(b)
One of the most widely used notions of negative dependence is NA (Negative Association), which is defined as follows.
First, two events $A, B \subset \Om$ are said to be orthogonal if there are two disjoint subsets $K, L \subset [n]$
such that $A$ is defined in terms of the indices in $K$ and $B$ in terms of the indices in $L$.
Now, a measure $\mu$ on $\Om$
is said to be NA if for all increasing, orthogonal events $A, B \subset \{0,1\}^n$,
$\mu(A \cap B) \leq \mu(A) \mu(B)$.

Note that if two events $A$ and $B$ are orthogonal, then clearly $A \square B = A \cap B$.
Hence BK implies NA. The reverse is not true (see \cite{M09}).

In the last twelve years there has been a lot of research activity aiming at a general theory of
negative dependence. This started with the papers \cite{P00} and \cite{DR98}. The understanding
of NA has enormously increased, in particular by an algebraic/(complex-)analytic
approach involving the zeroes of the generating polynomials (see \cite{B07}, 
\cite{BBL09}, \cite{BJ11}). Other techniques to study NA-related problems can be found in
\cite{KN10} and \cite{DJR07}.
However, so far these approaches do not work for the BK property
and it is unclear how this property would fit in a general framework. \\
(c) For some non-product measures, in particular Ising models, the following question makes sense: 
can the $\square$-operation be
{\em modified} in a natural way such that \eqref{bkr} holds for {\em all} events? 
This is investigated in \cite{BG11}.

\end{subsection}
\begin{subsection}{Reimer's intermediate result} \label{sect-reimer}
The following result, Proposition \ref{reimer} below, is essentially, but in different terminology,
Theorem 1.2 (or the equivalent Theorem 1.3) in \cite{R00}. 

As before, $\Om$ denotes $\{0,1\}^n$.
Some more notation is needed: For $\om = (\om_1, \cdots, \om_n) \in \Om$, we denote by $\bar\om$ the configuration
obtained from $\om$ by replacing $1$'s by $0$'s and vice versa:
$$\bar\om = (1-\om_1, \cdots, 1-\om_n).$$
Further, for $A \subset \Om$, we define
$\bar A \, = \, \{\bar\om \, : \, \om\in A\}$.
Finally, if $V$ is a finite set, $|V|$ denotes the number of elements of $V$.
Now we state Reimer's `intermediate' result to which we referred before:

\begin{prop}\label{reimer} {\em [Reimer \cite{R00}]} \\
For all $n$ and all $A, B \subset \{0,1\}^n$,
\begin{equation}\label{reimer-eq}
|A \square B| \leq |A \cap \bar B|.
\end{equation}
\end{prop}

\noindent
{\bf Remarks}: \\
(a) The very ingenious, linear-algebraic, proof of this proposition was the crucial part of Reimer's 
paper \cite{R00}.
The fact that
a result of the form of this proposition would imply Theorem \ref{bkr}
had already been discovered independently by other researchers (see \cite{T94}) . \\
(b) The language/terminology in \cite{R00} is somewhat unusual (for probabilists). This makes it, at first
sight, difficult to see that Proposition \ref{reimer} above is indeed equivalent to Theorem 1.2 in \cite{R00}.
Several authors have reviewed Reimer's paper with additional explanation
(see \cite{BCR99}).

\end{subsection}
\end{section}

\begin{section}{Proof of Theorem \ref{bkrx}} \label{sect-proof-bkrx} 
\begin{subsection}{An analog of Proposition \ref{reimer} for $k$-out-of-$n$ measures} \label{butx}
Let $k \leq n$, and recall the notation in Section \ref{sect-defs}.
The key to Theorem \ref{bkrx} is the (proof of the) following `analog' of Proposition \ref{reimer}:

\begin{prop}\label{reimerx}
For all even $m$ and all increasing $A, B \subset \{0,1\}^m$,
\begin{equation}\label{reimerx-eq}
P_{\frac{m}{2},m}(A \square B) \leq P_{\frac{m}{2},m}(A \cap \bar B).
\end{equation}
\end{prop}

We will show in Section \ref{reimx-1} that Theorem \ref{bkrx} follows from Proposition \ref{reimerx}.
Finally, in Section \ref{sect-proof-reimerx}, we will prove Proposition \ref{reimerx} 
by writing $P_{\frac{m}{2},m}$ as a suitable convex combination of measures for
which the analog of \eqref{reimerx-eq} can be derived from Proposition \ref{reimer} via a suitable
`encoding'. 

\end{subsection}

\begin{subsection}{Proof that Proposition \ref{reimerx} implies Theorem \ref{bkrx}} \label{reimx-1}
\begin{proof}
The proof below is quite similar to the proof (in \cite{R00}) that Proposition \ref{reimer} implies Theorem \ref{bkrx}.

Let $A$ and $B$ $\subset \Om$ be increasing, and let $k \leq n$.
We first rewrite the desired inequality, \eqref{bkrx-eq}, in an obvious way:

\begin{equation} \label{bkrx-equiv-eq}
(P_{k,n} \times P_{k,n}) ((A \square B) \times \Om) \leq (P_{k,n} \times P_{k,n}) (A \times B).
\end{equation}

\noindent
For each $K \subset [n]$ and $\al \in \{0,1\}^K$, define the `cell'

$$W_{\al} = \{ (\om, \om') \in \Om_{k,n} \times \Om_{k,n} \, :
\, \om_K = \om'_K = \al, \, \om_{K^c} = \overline{\om'_{K^c}}\},$$
where $K^c$ denotes $[n] \setminus K$.

\noindent
It is easy to see that these cells form a partition of $\Om_{k,n} \times \Om_{k,n}$. Hence it is sufficient to prove that,
for each $\al$ of the form mentioned above

\begin{equation} \label{bkrx-eq-cell}
|\left(\left(A \square B\right) \times \Om\right) \cap W_{\al}| 
\leq |\left(A \times B\right) \cap W_{\al}|.
\end{equation}

\noindent
Using the notation $|\al| = \sum_{i=1}^{n} \al_i$, notice that if $W_{\al} \neq \emptyset$ then
(since for every $(\om,\om') \in W_{\al}$ one has $2 k = |\om| + |\om'| = 2 |\al| + n - |K|$),

\begin{equation} \label{al-cond}
|\al| = k - (n - |K|)/2,
\end{equation}
and hence

\begin{equation} \label{om-cond}
|\om_{K^c}| = |\om'_{K^c}| = (n - |K|)/2 \, \mbox{ for all } (\om, \om') \in W_{\al}.
\end{equation}


Before going on, we introduce more notation.
Let $\Om_{(K^c)}$ be the set of all $\om \in \{0,1\}^{K^c}$ for which the
number of $1$'s and the number of $0$'s are equal (and hence equal to $(n- |K|)/2$). 
Define, for $\gamma \in \{0,1\}^{K^c},$ 
$\gamma \circ \alpha$ as the element of $\Om$ for which

$$(\gamma \circ \alpha)_{K^c} = \gamma, \text{ and } (\gamma\circ \alpha)_K = \alpha.$$

\noindent
Further, define for every event
$H \subset \Om$,

$$H(\alpha) = \{\gamma \in \{0,1\}^{K^c} \, : \,  \gamma\circ\alpha \in H\}.$$

From now on we assume, without loss of generality, that $\al$ satisfies \eqref{al-cond}.
Now suppose that $(\om,\om')$ belongs to the set in the r.h.s. of \eqref{bkrx-eq-cell}.
This holds if and only if $\om_K = \om'_K = \al$, $\om_{K^c} \in \Om_{(K^c)}$, 
$\om \in A$, $\om' \in B$ and $\om_{K^c} = \overline{\om'_{K^c}}$.


\noindent
The number of pairs $(\om, \om')$ that satisfy this is clearly
$|A(\al) \cap \overline {B(\al)} \cap \Om_{(K^c)}|$.
Similarly, it is easy to see that the l.h.s. of \eqref{bkrx-eq-cell} is equal to $|(A \square B)(\al) \cap 
\Om_{(K^c)}|$.
Further, it is easy to check from the $\square$-definition that

\begin{equation}\label{AB-eq}
(A \square B)(\al) \subset A(\al) \square B(\al)
\end{equation}

\noindent
So the l.h.s. of \eqref{bkrx-eq-cell} is at most $|(A(\al) \square B(\al)) \cap \Om_{(K^c)}|$. Hence, sufficient for 
\eqref{bkrx-eq-cell} to hold is

\begin{equation} \label{reimx-eq}
|(A(\al) \square B(\al)) \cap \Om_{(K^c)}|
\leq |A(\al) \cap \overline {B(\al)} \cap \Om_{(K^c)}|.
\end{equation}

\noindent
Finally, note that this last inequality follows immediately from Proposition \ref{reimerx}. (Replace the $m$ in
\eqref{reimerx-eq} by $ n - |K|$, and replace $A$ and $B$ by $A(\al)$ and $B(\al)$ respectively;
note that $A(\al)$ and $B(\al)$ are increasing because $A$ and $B$ are increasing).
This completes the proof that Proposition \ref{reimerx} implies Theorem \ref{bkrx}. \end{proof}

\end{subsection}

\begin{subsection}{Proof of Proposition \ref{reimerx}} \label{sect-proof-reimerx}
We first state and prove Proposition \ref{reimerc} below. Let $m$ be even. Let $\hOmn$
be the set of all $\om \in \{0,1\}^{ m}$ with the property that, for all $1 \leq i \leq m/2$,
$(\om_{2 i -1}, \om_{2 i})$ is equal to $(1,0)$ or $(0,1)$. Let $\hPn$ be the probability distribution on
$\{0,1\}^m$ which assigns equal probabilities to all $\om \in \hOmn$, and probability $0$ to all other $\om$.

The following is, as we will see, an `encoded form' of Proposition \ref{reimer}.

\begin{prop}\label{reimerc}
For all even $m$ and all increasing $A, B \subset \{0,1\}^{m}$,
\begin{equation}\label{reimerc-eq}
\hPn(A \square B) \leq \hPn(A \cap \bar B).
\end{equation}
\end{prop}

\begin{proof}
Let $A, B \subset \{0,1\}^{m}$ be increasing.
Note that \eqref{reimerc-eq} is equivalent to 

\begin{equation} \label{reimerc-eq-eq}
|(A \square B) \cap \hOmn| \leq |A \cap \bar B \cap \hOmn|.
\end{equation}

\noindent
Consider the following bijection $T: \hOmn \rightarrow \{0,1\}^{\frac{m}{2}}$:

$$T(\om_1, \cdots, \om_m) = (f((\om_1, \om_2), f(\om_3, \om_4), \cdots, f(\om_{m -1}, \om_m)),$$
where $f(1,0) = 1$ and $f(0,1) = 0$. 

\noindent
We claim that

\begin{equation}\label{check1}
T((A \square B) \cap \hOmn ) \subset  T(A \cap \hOmn) \,\, \square \,\, T(B \cap \hOmn),
\end{equation}
and that 

\begin{equation}\label{check2}
 T(A \cap {\bar B} \cap \hOmn) = T(A \cap \hOmn) \cap \overline{ T(B \cap \hOmn)}.
\end{equation}

\noindent
The first part of this claim, the inclusion \eqref{check1}, can be seen as follows: Let 
$\om = (\om_1, \cdots, \om_n) \in (A \square B) \cap \hOmn.$ 
By the definition of the $\square$-operation, there are $K, L \subset [n]$ such that 
$K \cap L = \emptyset$, $[\om]_K \subset A$, and $[\om]_L \subset B$. It is easy to see that this implies

\begin{equation} \label{T-arg}
[T(\om)]_{T(K)} \subset T(A \cap \hOmn), \, \mbox{ and } 
[T(\om)]_{ T(L)} \subset T(B \cap \hOmn),
\end{equation}
where $T(K) = \{\lceil i/2 \rceil \, : \, i \in K\}$ and $T(L)$ is defined analogously.
So far, the argument holds for all events. However, since $A$ and $B$ are increasing, we can even find
$K$ and $L$ such that, on top of the above properties, $\om \equiv 1$ on $K$ and $\om \equiv 1$ on $L$,
and hence, since $\om \in \hOmn$, 
\begin{equation} \label{T-arg2}
\mbox{ For all } 1 \leq i \leq \frac{m}{2}, \, K \cap \{2 i -1, i\} = \emptyset \mbox{ or }
L \cap \{2 i -1, i\} = \emptyset.
\end{equation}
From \eqref{T-arg2} it follows immediately that $ T(K) \cap T(L) = \emptyset$,
which, together with \eqref{T-arg},
gives 
$$T(\om) \in T(A \cap \hOmn) \,\, \square \,\, T(B \cap \hOmn),$$
completing the proof of \eqref{check1}. We omit the proof of \eqref{check2} (which is straightforward).

Now, using, in this order, \eqref{check1}, Proposition \ref{reimer} and \eqref{check2}, immediately gives 
\eqref{reimerc-eq-eq}. 
This completes the proof of Proposition \ref{reimerc}.
\end{proof}

\smallskip\noindent
{\bf Remark:}
At first sight one may think that in
\eqref{check1} even equality holds, but this is false: Take $n = 4$, $A = \{\om_1 = 1\} \cap \{\om_3 =1 
\mbox{ or } \om_4 = 1\}$, and $B = \{\om_3 =1\} \cap \{\om_1 =1 \mbox{ or } \om_2 =1\}$.
Then $A \square B$ (and hence the l.h.s. of \eqref{check1}) is $\emptyset$, while
the r.h.s. of \eqref{check1} is the subset of $\{0,1\}^2$ which contains only the element $(1,1)$.


\medskip
Now we are ready to prove Proposition \ref{reimerx}:

\begin{proof}
Let $m$ be even, and recall the definition of $P_{\frac{m}{2},m}$.
Let $A, B \subset \{0,1\}^m$ be increasing.
Let, for $\pi$ a permutation of $[m]$, $\hOmnp$ be the set of all $\om \in \{0,1\}^m$ with the
property that, for each $1 \leq i \leq m/2$, 
$(\om_{\pi(2 i -1)}, \om_{\pi(2 i)})$ is equal to $(0,1)$ or $(1,0)$.  Let $\hPnp$ be the probability distribution
on $\{0,1\}^m$ which assigns equal probabilities to all $\om \in \hOmnp$ and probability $0$ to all other
$\om \in \{0,1\}^m$. It follows immediately from Proposition \ref{reimerc} (by relabelling the indices) that for each $\pi$

\begin{equation}\label{reimercp-eq}
\hPnp(A \square B) \leq \hPnp(A \cap \bar B).
\end{equation}

Now observe that if we first randomly (and uniformly) draw a permutation $\pi$ and then randomly
draw a configuration $\om$ according to the distribution $\hPnp$, then $\om$ is a `typical' random configuration
drawn according to $P_{\frac{m}{2},m}$.
So $P_{\frac{m}{2},m}$ is a convex combination of the $\hPnp$'s. Finally, since each $\hPnp$ satisfies
\eqref{reimercp-eq}, every convex combination of the $\hPnp$'s also satisfies \eqref{reimercp-eq}. Hence
\eqref{reimerx-eq} holds. This completes the proof of Proposition \ref{reimerx}, and hence the proof
of Theorem \ref{bkrx}.

\end{proof}

\end{subsection}
\end{section}

\begin{section}{Some extensions of Theorem \ref{bkrx}} \label{sect-ext}
\begin{subsection}{Weighted $k$-out-of-$n$ measures} \label{sect-weight}

Let, as before, $1 \leq k \leq n$, and
recall the definition of $\Om_{k,n}$ and the $k$-out-of-$n$ measure $P_{k,n}$. 
Let $w_1, \cdots, w_n$ be non-negative numbers, let $w = (w_1, \cdots, w_n)$, and define the
probability measure $P_{k,n}^w$ (a weighted version of $P_{k,n}$ which is sometimes called
Conditional Poisson measure)) as follows:

$$P_{k,n}^w(\om) = C \, \mbox{I}(\om \in \Om_{k,n})\, \prod_{i=1}^n w_i^{\om_i}, \,\,\, \om \in \{0,1\}^n,$$
where $C$ is a normalizing constant and I denotes indicator function.
It is not difficult to see that the analog of Theorem \ref{bkrx}
holds for the weighted measures defined above. In fact, the proof remains practically the same by the following
observation: Let $W_{\al}$ be a cell (as in Section \ref{reimx-1}) with $\al$ satisfying \eqref{al-cond}, and
let $(\om, \om') \in W_{\al}$.
Since, for each index $i \in K^c$, exactly one of $\om_i$ and $\om'_i$ equals $1$ and the other equals $0$,
each index $i \in K^c$ contributes exactly a factor $w_i$ to the $(P_{k,n}^w \times P_{k,n}^w)$ measure. Moreover,
each $i \in K$ contributes (by the definition of $W_{\al}$) exactly a factor $(w_i)^{2 \alpha_i}$.
Hence, the proof again reduces to showing \eqref{reimx-eq}.

\end{subsection}

\begin{subsection}{Products of (weighted) $k$-out-of-$n$ measures} \label{sect-prod}
The proof of the BK inequality for increasing events under $k$-out-of-$n$ measures extends straightforwardly to
that for products of such measures:
By the arguments of Section \ref{reimx-1} the proof reduces to showing that Proposition \ref{reimerx}
holds for products of measures of the form $P_{\frac{m}{2},m}$. Now recall from Section \ref{sect-proof-reimerx} that
the reason that Proposition \ref{reimerx} holds
is, essentially, that $P_{\frac{m}{2},m}$ is a convex combination of measures of the form $\hPnp$. Now,
of course, products of measures $P_{\frac{m_1}{2}, m_1}, \cdots P_{\frac{m_l}{2}, m_l}$, are convex combinations
of measures of the
form $\hat{P}_{M,\pi}$, where $M = m_1 + \cdots + m_l$, and $\pi$ is a permutation
of $[M]$. Hence the proof goes through as before.
The above argument also goes through for products of {\em weighted} $k$-out-of-$n$ measures.

\end{subsection}

\begin{subsection}{Some ideas for further generalizations} \label{sect-ideas}
With $w = (w_1, \cdots, w_n)$ as in Section \ref{sect-weight}, and $X$ a random variable taking values in
$\{0, \cdots, n\}$, define $P_{X,n}^w$ as the measure of the configuration $\omega$ resulting from the
following procedure. First draw a number $k$ from the same distribution as $X$. Then draw an 
$\om \in \Om_{k,n}$ according to the distribution $P_{k,n}^w$. Motivated by the search for other examples
of BK measures it is natural to ask: for which $X$ is $P_{X,n}^w$ BK?
Of course, by Section \ref{sect-weight}, this is the case if $X$ is with probability $1$ equal to some constant $k$. 
Other examples of such $X$ can be easily obtained from the result in Section \ref{sect-weight} by adding
`dummy' indices and then projecting:
Let $m \geq 0$, and introduce auxiliary weights  $w_{n+1}, \cdots w_{n+m} \geq 0$.
Let $0 \leq k \leq n+m$. From section \ref{sect-weight}
we have that $P_{k, n+m}^{(w_1, \cdots, w_{n+m})}$ (a measure on $\{0,1\}^{n+m}$) is BK.
From the definition of BK it follows immediately that the BK property is preserved under projections.
Hence, the projection of $P_{k, n+m}^{(w_1, \cdots, w_{n+m})}$ on $\{0,1\}^n$ is also BK.
In other words, if we let, for a random configuration $(\om_1, \cdots \om_{n+m})$ drawn under
$P_{k, n+m}^{(w_1, \cdots, w_{n+m})}$,
the random variable $X$ denote $\sum_{i=1}^n \om_i$, then $P_{X,n}^w$ is BK.
It is not hard (but also, at this stage, not very helpful) to write a general form
for the distribution of an $X$ of this type. It would be interesting to find `natural' random variables $X$
which are not of this type but yet have the property that $P_{X,n}^w$ is BK.

\end{subsection}
\end{section}

\bigskip\noindent
{\bf\large Acknowledgments} \\
We thank an anonymous referee for many detailed comments on an earlier
version of this paper.
We also thank Ronald Meester, whose questions revived our interest in
these problems. J.vdB. thanks Demeter Kiss for his comments during a private
presentation of the results.  J.J. thanks Matthijs Joosten for comments on
earlier work on this subject, and Klas Markstr\"om for stimulating
communication.

\end{document}